\documentclass{amsart}
\usepackage{amsmath}
\usepackage{mathrsfs}
\usepackage{leftidx}
\usepackage{amssymb}
\usepackage{latexsym}
\usepackage[all]{xy}
\usepackage{graphicx}
\usepackage{caption}

\newtheorem{theorem}{Theorem}[section]

\newtheorem{proposition}[theorem]{Proposition}
\newtheorem{corollary}[theorem]{Corollary}
\theoremstyle{definition}
\newtheorem{definition}[theorem]{Definition}

\newtheorem{remark}[theorem]{Remark}
\numberwithin{equation}{section}

\begin{document}
\title{BLOCK SYSTEMS OF FINITE DIMENSIONAL NON-COSEMISIMPLE HOPF ALGEBRAS}
\author{Z.-P. Fan}
\address{Department of Mathematics, Zhejiang University,
Hangzhou, 310027, P.R. China} \email{fman1988@gmail.com}
\author{D.-M. Lu}
\address{Department of Mathematics, Zhejiang University,
Hangzhou, 310027, P.R. China} \email{dmlu@zju.edu.cn}

\subjclass[2000]{16T05, 16T15}



\keywords{the classification of finite dimensional Hopf algebras, coalgebra}

\begin{abstract}
In this paper, we study finite dimensional non-cosemisimple Hopf algebras through its underlying coalgebra. We decompose such a Hopf algebra as a direct sum of `blocks'. Blocks are closely related to each other by `rules' and form the `block system'. Through the block system, we are able to give a lower bound for $\dim H$, where $H$ is a non-cosemisimple Hopf algebras with no nontrivial skew-primitives, and a series of results about non-cosemisimple Hopf algebras of dimension $12p, 15p, 16p, 20p\ \text{and}\  21p$ with $p$ being a prime number.
\end{abstract}
\maketitle

\section*{Introduction}
Throughout the paper, we fix an algebraically closed field $\Bbbk$ with $\mathrm{char}\ \Bbbk= 0$. Every vector space is finite dimensional over $\Bbbk$ if not specified. For a coalgebra $C$, $\Delta$, $\epsilon$ and $G(C)$ denote comultiplication, counit and the set of group like elements respectively. For a Hopf algebra, $\mathcal{S}$ denotes the antipode.

The question of classifying all Hopf algebras of a given dimension comes from the Kaplansky's ten conjectures. So far there are only a few general results.
The classification splits into several different parts according to the coradical being cosemisimple, pointed or non-pointed non-cosemisimple. Recall that the coradical of a coalgebra is the direct sum of all its simple subcoalgebras. A coalgebra is called `cosemisimple' if it equals its coradical, and it is called `pointed' if every simple subcoalgebra is $1$-dimensional. Hopf algebras which are both pointed and cosemisimple are group algebras. Then in the following discussions, by `pointed' we mean `pointed but not cosemisimple'.

For a finite dimensional Hopf algebra, by \cite{Larson1988Finite}, it is semisimle as an algebra if and only if it is cosemisimple as a coalgebra. There are a few important results on cosemisimple Hopf algebras, but there is no general strategy for their classification so far. For $p,q,r$ being odd primes, it is already known that cosemisimple Hopf algebras with dimension $p,2p,p^2,pq$ are trivial (\cite{zhu1994hopf},\cite{masuoka1995semisimple},\cite{masuoka1996pn},\cite{etingof1998semisimple}), and the classification of cosemisimple Hopf algebras with dimension $p^3,2p^2,pq^2,pqr$ are completed (\cite{masuoka1995semisimple68},\cite{masuoka1995self},\cite{natale1999semisimple},\cite{natale2001semisimple},\cite{natale2004semisimple},\cite{etingof2011weakly}).

At the moment, the most general method for the classification of pointed Hopf algebras is the lifting method developed by Andruskiewitsch and Schneider. First they decompose the associated graded Hopf algebra $gr(A)$ of a pointed Hopf algebra $A$ into a smash biproduct of a braided Hopf algebra $R$ and the group algebra $\Bbbk G(A)$, then study the structure of $R$ as a Nichols algebra over $\Bbbk G(A)$, finally lift relations in $R$ to get ones in $A$ (\cite{Andruskiewitsch2002Pointed},\cite{Andruskiewitsch1998Lifting},\cite{andruskiewitsch2000finite},\cite{Andruskiewitsch2010On}).

For $p,q$ being odd primes and $p<q\le 4p+11$, there is no non-cosemisimple non-pointed Hopf algebra with dimension $p,2p,p^2,pq$ (\cite{zhu1994hopf},\cite{ng2005hopf},\cite{Ng2001Non},\cite{ng2008hopf}).
For non-cosemisimple non-pointed Hopf algebras of other dimensions, there are several classification methods based on different subjects. The dimension of spaces related to the coradical (\cite{andruskiewitsch2001counting},\cite{beattie2004hopf},\cite{fukuda2008structure}), the order of the antipode (\cite{Ng2001Non},\cite{ng2004hopf},\cite{ng2005hopf},\cite{ng2008hopf},\cite{hilgemann2009hopf}), the Hopf subalgebras and the quotient Hopf algebras (\cite{natale2002hopf}) and the braided Hopf algebras in smash biproducts (\cite{cheng2011hopf}) all play important roles in the classification of non-cosemisimple non-pointed Hopf algebras.

The idea of block system is inspired by the classification technique used in \cite{andruskiewitsch2001counting}, \cite{beattie2004hopf} and \cite{fukuda2008structure}. For a Hopf algebra, we build the block system based on its coradical filtration. Block systems become interesting when the Hopf algebras are non-cosemisimple and non-pointed. In this paper, we focus on non-cosemisimple Hopf algebras with no nontrivial skew primitives, which form a subclass of non-cosemisimple non-pointed Hopf algebras. Corollary \ref{cor1}, \ref{cor2}, \ref{cor3}, \ref{cor4} and Proposition \ref{prop3} are the `rules' we know so far, which lead to our main results, Theorem \ref{thm1} and Theorem \ref{thm2}. Theorem \ref{thm1} shows a lower bound for the dimensions of non-cosemisimple Hopf algebras with no nontrivial skew primitives. By Radford's formula for $\mathcal{S}^4$, orders of group like elements are crucial in classification methods working with the antipode. So Theorem \ref{thm2} contributes to the classification of Hopf algebras of dimensions involved in it.

The block system shows a way to understand a non-cosemisimple Hopf algebra through its coalgebra structure by decomposing it into blocks which are highly related by the `rules'. We believe that by further study, more `rules' will be discovered, so more clearly we understand the Hopf algebra structure.

\section{Block System}

We first recall some useful results on dimensions of Hopf algebras related to the coradical filtration (\cite{{andruskiewitsch2001counting}, {beattie2004hopf}, {beattie2013techniques}, {fukuda2008structure}}).

For a coalgebra $C$, we denote by $\{C_n\}_{n\in \mathbb{N}}$ its coradical filtration.

For $x\in C$ and $g,h\in G(C)$, if $\Delta (x)= g\otimes x + x\otimes h$ and $x\notin \Bbbk (g-h)$, then $x$ is called \textit{a nontrivial skew primitives} of $C$.

By \cite[Theorem 5.4.2]{montgomery1993hopf}, there exists a coalgebra projection $\pi: C\to C_0$ with kernel $I$. Set $\rho_L = (\pi \otimes id)\Delta$ and $\rho_R = (id \otimes\pi)\Delta$. Then $C$ becomes a $C_0$-bicomodule with the structure maps $\rho_L$ and $\rho_R$.

Let $P_n$ be the sequence of subspaces defined recursively by

\begin{align*}
P_0 &=0, \\
P_1 &=\{c\in C \ | \  \Delta(c)=\rho_L(c)+\rho_R(c)\}= \Delta^{-1}(C_0\otimes I + I \otimes C_0), \\
P_n &=\{c\in C\ | \ \Delta(c)-\rho_L(c)-\rho_R(c)\in \sum_{1\le i \le n-1} P_i \otimes P_{n-i}\}, \ \ n\ge 2.
\end{align*}
By \cite[Lemma 1.1]{andruskiewitsch2001counting}, $P_n= C_n \cap I$.

For a positive integer $d$, let $D$ be a simple coalgebra of dimension $d^2$. There is \textit{a standard basis} $\{e_{i,j}\}_{i,j\in\{1,\cdots,d\}}$ of $D$, such that 
$$
\Delta(e_{i,j})=\sum_{k=1}^d e_{i,k}\otimes e_{k,j} \ \text{and}\  \varepsilon(e_{i,j})=\delta_{i,j}. $$
The $j$th column $\{e_{i,j}\}_{i\in \{1,\cdots,d\}}$ is a basis of the simple left $D$-comodule $\bigoplus_{i\in \{1,\cdots,d\}} \Bbbk e_{i,j} $ with comultiplication as comodule map. Every simple left $D$-comodule is isomorphic to $\bigoplus_{i\in \{1,\cdots,d\}} \Bbbk e_{i,j} $. Similarly, the $i$th row $\{e_{i,j}\}_{j\in \{1,\cdots,d\}}$ forms a basis of the unique simple right $D$-comodule up to isomorphisms.

Let $\hat{C}$ be the set of isomorphism types of simple left $C$-comodules. There is a one to one correspondence between $\hat{C}$ and the set of simple subcoalgebras of $C$. Every simple left $C$-comodule is isomorphic to the unique (up to isomorphisms) simple left comodule of a simple subcoalgbra of $C$. For $\tau \in \hat{C}$, let $D_\tau$ be the corresponding simple subcoalgebra of $C$ with $\dim D_\tau= {d_{\tau}}\!\!^2$. Then we have
$$
C_0= \bigoplus_{\tau\in \hat{C}} D_{\tau}.
$$
For convenience, we denote by $D_g$ the simple subcoalgebra generated by $g\in G(C)$, and the index of $D_g$ in $\hat{C}$ is also $g$.

Set $C_{0,d} = \bigoplus_{\tau\in \hat{C},d_\tau = d} D_\tau$ with $d\ge 1$. Easy to see that $\Bbbk G(C)= C_{0,1}$.

See that the dual of a simple left $C$-comodule is also a simple right $C$-comodule. For $\tau \in \hat{C}$, we denote by $V_\tau$ (respectively, $V_\tau^*$) the simple left (respectively, right) $C$-comodule corresponding to the simple subcoalgebra $D_\tau$ (as a representative of the unique isomorphism type). Let $\{e^\tau_{i,j}\}_{i,j\in\{1,\cdots,d_\tau\}}$ be a standard basis of $D_\tau$, then
$$
V_\tau \simeq \bigoplus_{i\in \{1,\cdots,d\tau\}} \Bbbk e^\tau_{i,j}\ \text{and}\ V_\tau^* \simeq \bigoplus_{j\in \{1,\cdots,d\tau\}} \Bbbk e^\tau_{i,j}.
$$

Since $C_0$ is cosemisimple and $\Bbbk$ is a perfect filed, any $C_0$-bicomodule can be decomposed as a direct sum of simple $C_0$-subbicomodules. It is well known that every simple $C_0$-bicomodule is of the form $V_\tau \otimes V_\mu^*$ and of dimension $d_\tau d_\mu$, for some $\tau,\mu \in \hat{C}$ with $\dim V_\tau=d_\tau$ and $dim V_\mu^* =d_\mu$. 

Through $\rho_L$ and $\rho_R$, $C_n$ and $P_n$ are $C_0$-subbimodules of $C$. As in \cite{{beattie2004hopf},{beattie2013techniques},{fukuda2008structure}}, for $\tau,\mu \in \hat{C}$, $P^{\tau,\mu}_n$ denotes the direct sum of simple $C_0$-subbicomodules of the form $V_\tau \otimes V_\mu^*$ in $P_n$. $P^{\tau,\mu}_n$ is called \textit{non-degenerate} if $P^{\tau,\mu}_n\nsubseteq P_{n-1}$. See that $P^{\tau,\mu}_n$ is a $C_0$-bicomodule and $P^{\tau,\mu}_{n-1}\subseteq P^{\tau,\mu}_n$. Then there exists a $C_0$-bicomodule $Q^{\tau,\mu}_n$, which is also isomorphic to a direct sum of simple $C_0$-bicomodules of the form $V_\tau \otimes V_\mu^*$, such that
$$
P^{\tau,\mu}_n=P^{\tau,\mu}_{n-1}\oplus Q^{\tau,\mu}_n.
$$


Then we give the definition of a block system.
\begin{definition}
For a coalgebra $C$, we set
$$
B^{d_1,d_2}_n =
\left\{
\begin{array}{ll}
0,&n=0, d_1\ne d_2,\\
C_{0,d},&n=0, d_1=d_2=d\ge 1,\\
\bigoplus_{\tau,\mu \in \hat{C}, d_\tau=d_1, d_\mu=d_2}Q^{\tau,\mu}_n, &n, d_1 , d_2\ge1.
\end{array}\right.
$$
and call it a \textit{block} of $C$.

It is obvious that $C= \bigoplus_{n\ge0,d_1,d_2\ge1} B^{d_1,d_2}_n$, and we call this direct sum \textit{the block system of $C$}.
\end{definition}


Let $H$ be a Hopf algebra. Follow the notations above. See that $P^{\tau,\mu}_n$ is non-degenerate if and only if $Q^{\tau,\mu}_n \ne 0$ for $\tau,\mu \in \hat {H}$ and $n\ge 1$. Hence we can rewrite some results in \cite{beattie2013techniques} and \cite{fukuda2008structure}.

\begin{proposition}\label{prop1}
For $\tau,\mu \in \hat {H}$, we have the following results.
\begin{enumerate}
\item \cite[Lemma 3.2]{fukuda2008structure} If $Q^{\tau,\mu}_n \ne 0$ for some $n>1$, then there exists a set of simple subcoalgebras $\{D_1, \cdots,D_{n-1}\}$ such that $Q^{\tau,D_i}_i \ne 0$ and $Q^{D_i,\mu}_i \ne 0$ for all $1\le i\le n-1$.
\item \cite[Lemma 3.5]{fukuda2008structure} For $g\in G(H)$, $n\ge1$, $\dim Q^{\tau,\mu}_n = \dim Q^{\mathcal{S}\tau,\mathcal{S}\mu}_n =\dim Q^{g\tau,g\mu}_n =\dim Q^{\tau g,\mu g}_n $, where the superscript $\mathcal{S}\tau$ means that the simple subcoalgebra is $\mathcal{S}(D_\tau)$ and $g\tau$ (resp. $\tau g$) means the simple subcoalgebra $gD_\tau$ (resp. $D_\tau g$).
\item \cite[Lemma 3.8]{fukuda2008structure} Assume that $Q^{\tau,\mu}_n \ne 0$ for some $n\ge1$. If $d_\tau \ne d_\mu$ or $\dim Q^{\tau,\mu}_n\ne d_\tau^2$, then there exists a simple subcoalgebra $E$ such that $Q^{\tau,E}_{n'} \ne 0$ for some $n' \ge n+1$.
\end{enumerate}
\end{proposition}

\begin{proposition}\cite[Proposition 3.2 (i)]{beattie2013techniques}\label{prop2}
If $H$ is a non-cosemisimple Hopf algebra with no nontrivial skew-primitives, then for any $g\in G(H)$, there exist $h\in G(H)$ and simple subcoalgebras $D_\tau$ and $D_\mu$ with $\tau,\mu \in \hat {H}$ and $d_\tau,d_\mu > 1$, such that $d_\tau=d_\mu$, $Q^{\tau,g}_1\ne 0$, $Q^{\tau,\mu}_k\ne0$ for some $k>1$ and $Q^{g,h}_m\ne 0$ for some $m>1$.
\end{proposition}

The following results show properties of the blocks of a Hopf algebra, which should be considered as `rules' for block systems.

\begin{corollary}\label{cor1}
For blocks of $H$ and $g\in G(H)$, we have
$$
gB^{d_1,d_2}_n = B^{d_1,d_2}_n g= B^{d_1,d_2}_n, \ \text{and}\
\mathcal{S}(B^{d_1,d_2}_n ) = B^{d_2,d_1}_n,
$$
with $n\ge0$ and $d_1,d_2\ge1$.

\end{corollary}
\begin{proof}
The image of a simple subcoalgebra of $H$ under $\mathcal{S}$ or left (right) action of $g\in G(H)$ is still a simple one of the same dimension. Hence the proposition is correct for $n=0$.

$\mathcal{S}$, $\mathcal{S}^{-1}$ and left action of $g\in G(H)$ are bijections and keep the coradical filtration. For $n\ge1$, following the definition of $Q^{\tau,\mu}_n$, we have $gQ^{\tau,\mu}_n = Q^{g\tau,g\mu}_n$ and $\mathcal{S}(Q^{\tau,\mu}_n)=Q^{\mathcal{S}\mu,\mathcal{S}\tau}_n$. By Proposition \ref{prop1}(2), $gQ^{\tau,\mu}_n \subseteq B^{d_1,d_2}_n$ and $\mathcal{S}(Q^{\tau,\mu}_n)\subseteq B^{d_2,d_1}_n$  with $d_\tau=d_1, d_\mu=d_2$, which leads to $gB^{d_1,d_2}_n\subseteq B^{d_1,d_2}_n$ and $\mathcal{S}(B^{d_1,d_2}_n ) \subseteq B^{d_2,d_1}_n$. Similarly $g^{-1}B^{d_1,d_2}_n\subseteq B^{d_1,d_2}$ and $\mathcal{S}^{-1}(B^{d_1,d_2}_n ) \subseteq B^{d_2,d_1}_n$.
Hence $gB^{d_1,d_2}_n = B^{d_1,d_2}_n$ and $\mathcal{S}(B^{d_1,d_2}_n ) = B^{d_2,d_1}_n$.

Similar for $B^{d_1,d_2}_n g= B^{d_1,d_2}_n$.
\end{proof}

\begin{corollary}\label{cor2}
For blocks of $H$, $\dim B^{1,1}_0=|G(H)|$ divides the dimension of every block of $H$.
\end{corollary}
\begin{proof}
Let $B^{d_1,d_2}_n$ be a block of $H$. By Corollary \ref{cor1}, every block of $H$ is stable under the left multiplication of $g\in G(H)$.

For $n=0$, $B^{d_1,d_2}_0$ with $d_1=d_2$ is a left $(H,\Bbbk G(H))$-Hopf module with comultiplication of $H$ as comodule map. Then by the Nichols-Zoeller Theorem, $B^{d_1,d_2}_0$ is a free left $\Bbbk G(H)$-module. Then $|G(H)|$ divides the dimension of $B^{d_1,d_2}_0$.

For $n>0$, we have
\begin{align*}
\Delta(B^{d_1,d_2}_n)&\subset H_0\otimes (B^{d_1,d_2}_n)+  (B^{d_1,d_2}_n)\otimes H_0 +\sum_{1\le i \le n-1} P_i \otimes P_{n-i}\\
&\subset H_0\otimes (B^{d_1,d_2}_n)+  (B^{d_1,d_2}_n)\otimes H_0 + H_{n-1}\otimes H_{n-1}.
\end{align*}
See that $H_{n-1}$ and $B^{d_1,d_2}_n\oplus H_{n-1}$ are left $(H,\Bbbk G(H))$-Hopf modules with left multiplication of $\Bbbk G(H)$ as module map and comultiplication of $H$ as comodule map. Then we have $H_{n-1}$ and $B^{d_1,d_2}_n\oplus H_{n-1}$ are free over $\Bbbk G(H)$ and $|G(H)|$ divides the dimensions of $H_{n-1}$ and $B^{d_1,d_2}_n\oplus H_{n-1}$, which leads to the fact that $B^{d_1,d_2}_n$ is free over $\Bbbk G(H)$ and $|G(H)|$ divides the dimension of $B^{d_1,d_2}_n$.
\end{proof}

Every nonzero $Q^{\tau,\mu}_n$ leads to a nonzero block $B^{d_1,d_2}_n$ with $d_1 =d_\tau, d_2 =d_\mu$. Hence the following corollary is a direct result from Proposition \ref{prop1}.

\begin{corollary}\label{cor3}
For blocks of $H$, we have the following results.
\begin{enumerate}
\item If $B^{d_1,d_2}_n \ne 0$ for some $n>1$, then there exists a set of positive integers $\{b_1, \cdots,b_{n-1}\}$ such that $B^{d_1,b_i}_i \ne 0$ and $B^{b_i,d_2}_{n-i} \ne 0$ for all $1\le i\le n-1$.
\item If $B^{d_1,d_2}_n \ne 0$ for some $n\ge 1$, then $B^{d_2,d_1}_n \ne 0$.
\item Assume that $B^{d_1,d_2}_n \ne 0$ for some $n\ge1$. If $d_1 \ne d_2$, then there exists positive integers $d_3$ and $d_4$ such that $B^{d_1,d_3}_{l_1}, B^{d_4,d_2}_{l_2}\ne 0$ for some $l_1,l_2\ge n+1$.
\end{enumerate}
\end{corollary}

\begin{remark}\label{rmk2}
Since we always assume that $H$ is of finite dimension, then by Corollary \ref{cor3}(2),(3), $B^{d_1,d_2}_n\ne 0$ for $n\ge1$ with $d_1\ne d_2$ leads to  $B^{d_1,d_1}_{n_1}\ne 0$ and $B^{d_2,d_2}_{n_2}\ne 0$ for some $n_1,n_2\ge n+1$.
\end{remark}

Corollary \ref{cor4} is a direct result from Proposition \ref{prop2} and Corollary \ref{cor3} (2). It shows the necessary blocks for the block system of a non-cosemisimple Hopf algebra with no nontrivial skew-primitives.

\begin{corollary}\label{cor4}
If $H$ is a non-cosemisimple Hopf algebra with no nontrivial skew-primitives, then there exist simple subcoalgebras $D_\tau$ and $D_\mu$ of H with $\tau,\mu \in \hat {H}$ and $d_\tau=d_\mu>1$, such that $B^{d_\tau,1}_1$, $B^{1, d_\tau}_1$, $B^{d_\tau,d_\mu}_k$ and $B^{1,1}_m\ne 0$ for some $k,m>1$.
\end{corollary}

\begin{remark}\label{rmk1}
By \cite[Proposition 1.8]{andruskiewitsch2001counting}, a Hopf algebra has no nontrivial skew-primitives if and only if it has no nontrivial pointed Hopf subalgebras. By nontrivial pointed Hopf subalgebras we mean pointed Hopf sbualgebras which are not group algebras. Following \cite[Lemma 2.8]{beattie2013classifying}, if $(|G(H)|, \dim H / |G(H)|)=1$, then $H$ has no nontrivial skew-primitives. A special case is that $\dim H$ is free of squares.
\end{remark}

For a pointed Hopf algebra, all the blocks are of the form $B^{1,1}_n$ with $n\ge0$. For a non-cosemisimple Hopf algebra with no nontrivial nontrivial skew-primitives, $B^{1,1}_n$'s also play an important role.

\begin{proposition}\label{prop3}
For blocks of $H$, let $m= \max \{m'\ |\ B^{1,1}_{m'}\ne 0\}$ and  $l= \min \{m'\ |\ B^{1,1}_{m'}\ne 0\}$, then
\begin{enumerate}
\item $\dim B^{1,1}_m = |G(H)|$;
\item if $l< m$, and $H$ is non-cosemisimple and has no nontrivial skew-primitives, then there exist $d_1,d_2,d_3,d_4>1$ and $l'>l>1$, such that $B^{d_1,1}_{l'}\ne 0$, $B^{1,d_2}_{l'}\ne 0$, $B^{d_1,d_3}_{l'-1}\ne0$ and $B^{d_4,d_2}_{l'-1}\ne0$.
\end{enumerate}
\end{proposition}
\begin{proof}
(1) If $H$ is cosemisimple, then $m=0$ and $B^{1,1}_0=\Bbbk G(H)$.

If $H$ is non-cosemisimple, then $m\ge 1$ by Corollary \ref{cor4}.

We built a basis for $H$ at first.

For any $\sigma \in \hat{H}$, let $\{e^\sigma_{i_\sigma,j_\sigma}\}_{i_\sigma,j_\sigma\in\{1\,\cdots,d_\sigma\}}$ be a standard basis of $D_\sigma$.

Recall that
$$
H=\bigoplus_{n'\ge0,d_1,d_2\ge1} B^{d_1,d_2}_{n'}=\bigoplus_{d\ge1}H_{0,d} \ \oplus \bigoplus_{n>0,d_1,d_2\ge1}\Big( \bigoplus_{\tau,\mu \in \hat{H}, d_\tau=d_1, d_\mu=d_2}Q^{\tau,\mu}_n \Big),
$$
where $Q^{\tau,\mu}_n = \bigoplus_k V^{\tau,\mu}_{n,k}$ with $V^{\tau,\mu}_{n,k} \simeq V_\tau\otimes V_\mu^*$ as simple $H_0$-bicomudules for some finite index $k\in \{1,\cdots,k(n,\tau,\mu)\}$. Here $k(n,\tau,\mu)$ denotes the number of simple $H_0$-bicomudules of the form $V_\tau\otimes V_\mu^*$ in $Q^{\tau,\mu}_n$.

We have that $D_\tau$ is the simple subcoalgebra of $H$ corresponding to $V_\tau$ with $\dim D_\tau = {d_\tau}\!\!^2$. Through the isomorphism form the unique simple left $D_\tau$-comodule to $V_\tau$, $\{e^\tau_{i_\tau,1}\}_{i_\tau\in\{1\,\cdots,d_\tau\}}$ leads to a basis of $V_\tau$. Similarly, $\{e^\mu_{1,j_\mu}\}_{j_\mu\in\{1\,\cdots,d_\mu\}}$ leads to a basis of $V_\mu^*$. Then $\{e^\tau_{i_\tau,1}\otimes e^\mu_{1,j_\mu}\}_{i_\tau\in\{1\,\cdots,d_\tau\},{j_\mu\in\{1\,\cdots,d_\mu\}}}$ leads to a basis of $V_\tau\otimes V_\mu^*$, and through the isomorphism from $V_\tau\otimes V_\mu^*$ to $V^{\tau,\mu}_{n,k}$, we have a basis of $V^{\tau,\mu}_{n,k}$, denoted by $W^{\tau,\mu}_{n,k}$.

Set
$$
W= \Big(\bigcup_{\sigma \in \hat{H},i_\sigma,j_\sigma\in\{1\,\cdots,d_\sigma\}} \{e^\sigma_{i_\sigma,j_\sigma}\}\Big) \ \cup \  \Big(\bigcup_{\tau,\mu\in \hat{H},n>0,k\in \{1,\cdots,k(n,\tau,\mu)\}} W^{\tau,\mu}_{n,k}\Big).
$$
Then $W$ is a basis of $H$.

Then we show that there is a basis element in $B^{1,1}_m$ such that its dual is the left integral of $H^*$.

Following notations at the beginning of this section, for any $w'\in W-H_0$, we have
$$
\Delta(w')= \rho_L(w')+\rho_R(w') + \sum_{x,y\in W} k_{x,y} x\otimes y,\ \text{for $k_{x,y}\in \Bbbk$}.
$$
Set 
$$
\begin{array}{l}
L(w)=\left\{
\begin{array}{ll}
\{x\ |\ k_{x,y}\ne 0,y\in W\},& w\in W-H_0,\\
\phi, & w\in W\cap H_0,\ \text{and}\ 
\end{array}\right. \\
R(w)=\left\{
\begin{array}{ll}
\{y\ |\ k_{x,y}\ne 0,x\in W\},& w\in W-H_0,\\
\phi, & w\in W\cap H_0.
\end{array}\right.
\end{array}
$$

There exists $u\in W\cap B^{1,1}_m$, such that $\rho_L(u)+\rho_R(u)= 1\otimes u+ u\otimes g$ for some $g\in G(H)$.

We claim that $u\notin L(w)\cup R(w)$ for all $w\in W$. If $u\in R(w)$, for some $w\in W$, then by the proof of \cite[lemma 3.2]{fukuda2008structure}, $w\in B^{d,1}_r$, for $r>m$. By Corollary \ref{cor3} and Remark \ref{rmk2}, $B^{d,1}_r\ne0$ for $r>m$ leads to $B^{1,1}_{r'}\ne 0$ for $r'>m$, which is a contradiction with the maximality of $m$. Similar for $u\notin L(w)$.

Let $\{w^*\ |\ w\in W\}$ be the dual basis of $W$. Consider the convolution product $w^**u^*$. For any $v\in W$, by $u\notin R(v)$ and the definition of $W$, we have,
$$
w^**u^*(v)=\left\{
\begin{array}{ll}
1& w=1,v=u,\\
0& \text{otherwise}
\end{array}\right..
$$
Hence $w^**u^*=w^*(1)u^*$ for $w\in W$. Then for any $f\in H^*$, $f*u^*=f(1)u^*$, which makes $u^*$ a left integral of $H^*$.

Finally, we show that $\{hu\}_{h\in G(H)}$ is a basis of $B^{1,1}_m$.

See that $W\cap B^{1,1}_m$ is a basis of $B^{1,1}_m$. For $u'\in W\cap B^{1,1}_m$, assume that $\rho_L(u')+\rho_R(u')= g_1\otimes u'+ u'\otimes g_2$, $g_1,g_2\in G(H)$. Then there exists $k\in \Bbbk$ such that $kg_1^{-1}u' \in W$ and
$$
\rho_L(kg_1^{-1}u')+\rho_R(kg_1^{-1}u')= 1\otimes kg_1^{-1}u'+ kg_1^{-1}u'\otimes g_1^{-1}g_2.
$$
Since the space of left integral of $H^*$ is one-dimensional, we must have $kg_1^{-1}u'\in \Bbbk u$ and $g_1^{-1}g_2=g$. Then $u'=k'g_1u$ for some $k'\in\Bbbk$. Since $\rho_L+ \rho_R(\{hu\}_{h\in G(H)})$ is a linearly independent set in $H\otimes H$, $\{hu\}_{h\in G(H)}$ is also linearly independent. Then $\{hu\}_{h\in G(H)}$ is a basis of $B^{1,1}_m$ and $\dim B^{1,1}_m = |G(H)|$.

(2) Following notations in (1). Since $H$ has no nontrivial skew-primitives, we have $B^{1,1}_1 =0$ and $l>1$. By Corollary \ref{cor3} (1), the existence of $d_3,d_4$ with $B^{d_1,d_3}_{l'-1}\ne0$ and $B^{d_4,d_2}_{l'-1}\ne0$ is ensured by the existence of $d_1,d_2$.

For any $z\in W\cap B^{1,1}_l$, there exists $w\in W$ such that $z\in R(w)$ since $l<m$. Assume that $w\in B^{d,1}_t$.

If $d>1$, set $d_1=d_2=d$ and $l'=t$. By Corollary \ref{cor3} (2), $B^{1,d_2}_{l'}\ne0$. Then such $d_1,d_2$ satisfy the requirements.

Otherwise $d=1$. By the proof of \cite[lemma 3.2]{fukuda2008structure} and the minimality of $l$, we have $t\ge 2l$. For $B^{1,1}_t\ne0$, there exist $B^{1,d'}_{t-1},B^{d',1}_1 \ne 0$. Since $H$ has no nontrivial skew-primitives, we have $d'>1$. With $t-1\ge 2l-1 >l$, $B^{1,d'}_{t-1}$ and $B^{d',1}_{t-1}$ are not zero. Set $d_1=d_2=d'$ and $l'=t-1$. Then (2) is proved.
\end{proof}

\section{Applications}

By Corollary \ref{cor3} and Corollary \ref{cor4}, if $H$ is a non-cosemisimple Hopf algebra with no nontrivial skew-primitives, then its block system has the form as
$$
B^{1,1}_0 \oplus B^{d,d}_0\oplus B^{d,1}_1 \oplus B^{1,d}_1\oplus \cdots\oplus B^{1,1}_m \oplus B^{d,d}_{k},
$$
for some $d> 1$, $k= \max \{k'\ |\ B^{d,d}_{k'}\ne 0\}$ and $m= \max \{m'\ |\ B^{1,1}_{m'}\ne 0\}$. By Remark \ref{rmk2}, such a $k$ exits and if $B^{d,d'}_{n}\ne 0$ for some $d'\ge 1$, then $n\le k$. We use a diagram (Figure 1) to show the structure of the block system.

\unitlength=1cm
$$\begin{picture}(8.5,4.5)(0.1,0)
\put(0,0){\framebox(2.5,0.7)[]{$\mathbf{B^{1,1}_0}$}}
\put(3,0){\framebox(2.5,0.7)[]{$\mathbf{B^{d,d}_0}$}}
\put(3,1){\framebox(2.5,0.7)[]{$\mathbf{B^{d,1}_1}$}}
\put(6,1){\framebox(2.5,0.7)[]{$\mathbf{B^{1,d}_1}$}}
\put(3,3){\framebox(2.5,0.7)[]{$\mathbf{B^{1,1}_m}$}}
\put(6,3.5){\framebox(2.5,0.7)[]{$\mathbf{B^{d,d}_k}$}}
\put(4.3,2){\circle*{0.1}}
\put(4.3,2.3){\circle*{0.1}}
\put(4.3,2.6){\circle*{0.1}}
\put(7.3,2.3){\circle*{0.1}}
\put(7.3,2.6){\circle*{0.1}}
\put(7.3,2.9){\circle*{0.1}}
\put(6.8,0.4){\circle*{0.1}}
\put(7.3,0.4){\circle*{0.1}}
\put(7.8,0.4){\circle*{0.1}}
\end{picture}$$
\begin{center}
(Figure 1)
\end{center}
\vskip5mm

These six blocks are necessary for every non-cosemisimple Hopf algeabra with no nontrivial skew-primitives. If the block system of a coalgebra does not contain these six blocks above, then this coalgebra does not admits the structure of a non-cosemisimple Hopf algeabra with no nontrivial skew-primitives.

We find a lower bound for the dimension of a non-cosemisimple Hopf algebra with no nontrivial skew-primitives, which is a generation of \cite[Proposition 3.2]{beattie2013techniques}. For $n_1,n_2\in \mathbb{N}$, we denote by $lcm(n_1,n_2)$ the least common multiple of $n_1$ and $n_2$.

\begin{theorem}\label{thm1}
If $H$ is a non-cosemisimple Hopf algebra with no nontrivial skew-primitives and $|G(H)|=r$,  then
$$
\dim H \ge \min\{(2d+2)r+2lcm(d^2,r)\ |\ d>1\}.
$$
\end{theorem}
\begin{proof}
First we take the lower bound of the possible dimensions of each necessary block. Then add them together (see that these sums are related to $d$). Finally pick the minimal one for all $d>1$ as the lower bound of $\dim H$.

We have $\dim B^{1,1}_0 = |G(H)|=r$. Since $n$ divides the dimension of every block and $d^2$ divides $\dim H_{0,d}$, we have $\dim B^{d,d}_0=\dim H_{0,d}$ is a multiple of $lcm(d^2,r)$. Similarly $\dim B^{d,d}_k$ is a multiple of $lcm(d^2,r)$. By Proposition \ref{prop3} (1), $\dim B^{1,1}_m = |G(H)|=r$.

The dimension of a simple $H_0$-bicomodule in $B^{d,1}_1$ is $d$, and with left multiplications of group like elements in $G(H)$, we have another $|G(h)|-1$ simple $H_0$-bicomodules of different isomorphism types in $B^{d,1}_1$. Hence $\dim B^{d,1}_1$ is a multiple of $d|G(H)|=dr$. By Corollary \ref{cor1}, $\dim B^{1,d}_1 =\dim B^{d,1}_1$ is also a multiple of $dr$.
\end{proof}

For further discussions, we make some dotations at first.
\begin{definition}\label{def2}
For $r,d_1,d_2\ge1$,
\begin{enumerate}
\item set
$$
f(r,d_1,d_2)=\left\{
\begin{array}{ll}
r& d_1=1,d_2=1,\\
2d_2r& d_1=1,d_2\ne1,\\
2d_1r& d_1\ne1,d_2=1,\\
lcm(d_1d_2,n)& d_1\ne1,d_2\ne1;
\end{array}\right.
$$
\item set $L_{r,d_1,d_2}=\Bbbk^{\oplus f(r,d_1,d_2)}$ as a $\Bbbk$-space of dimension $f(r,d_1,d_2)$ and call it \textit{a basic block};
\item for $d>1$, set $L(r,d)=(L_{r,1,1})^{\oplus 2}\oplus (L_{r,d,d})^{\oplus 2}\oplus L_{r,d,1}$ and call it \textit{a minimal form}.
\end{enumerate}
\end{definition}

For $|G(H)|=r$ and $d,d_1,d_2>1$, $n\ge0$, we have
\begin{enumerate}
\item $\dim L(r,d) = (2d+2)r+2lcm(d^2,r)$; 
\item $\dim B^{1,1}_n$ is a multiple of $r=\dim L_{r,1,1}$;
\item $\dim B^{d,1}_n\oplus B^{1,d}_n$ is a multiple of $2dr=\dim L_{r,d,1}=\dim L_{r,1,d}$;
\item $\dim B^{d_1,d_2}_n$ is a multiple of $lcm(d_1d_2,r)=\dim L_{r,d_1,d_2}$.
\end{enumerate}
Then as $\Bbbk$-spaces, $B^{1,1}_n$, $B^{d,1}_n\oplus B^{1,d}_n$ and $B^{d_1,d_2}_n$ are isomorphic to a multiple of $L_{r,1,1}$, $L_{r,d,1}$ and $L_{r,d_1,d_2}$, respectively.

Assume that $H$ is non-cosemisimple and has no nontrivial skew-primitives. Then as $\Bbbk$-spaces, the block system of $H$ is isomorphic to a minimal form $L(|G(H)|,d)$ with more (or no) basic blocks being added. Of course, the adding of basic blocks into the minimal form must coincide with `rules' of block system.

Compare Figure 1 and Figure 2 to see the corresponding relations between necessary blocks and basic blocks.

\unitlength=1cm
$$
\begin{picture}(8.5,4.5)(0.1,0)
\put(0,0){\framebox(2.5,0.7)[]{$\mathbf{L_{r,1,1}}$}}
\put(3,0){\framebox(2.5,0.7)[]{$\mathbf{L_{r,d,d}}$}}
\put(3,1){\framebox(5.5,0.7)[]{$\mathbf{L_{r,d,1}}$}}
\put(3,3){\framebox(2.5,0.7)[]{$\mathbf{L_{r,1,1}}$}}
\put(6,3.5){\framebox(2.5,0.7)[]{$\mathbf{L_{r,d,d}}$}}
\put(4.3,2){\circle*{0.1}}
\put(4.3,2.3){\circle*{0.1}}
\put(4.3,2.6){\circle*{0.1}}
\put(7.3,2.3){\circle*{0.1}}
\put(7.3,2.6){\circle*{0.1}}
\put(7.3,2.9){\circle*{0.1}}
\put(6.8,0.4){\circle*{0.1}}
\put(7.3,0.4){\circle*{0.1}}
\put(7.8,0.4){\circle*{0.1}}
\end{picture}
$$
\begin{center}
(Figure 2)
\end{center}
\vskip5mm

With more condition on $|G(H)|$, we have the next result following Proposition \ref{prop3} and Theorem \ref{thm1}. See that $|G(H)|$ divides $\dim H$.

\begin{corollary}\label{cor5}
If $H$ is a non-cosemisimple Hopf algebra with no nontrivial skew-primitives and $|G(H)|=p$ with $p$ being a prime number, set $\dim H=tp$ with $t\ge 1$, then we have
\begin{enumerate}
\item if $p=2$, then $t\ne 1,2,\cdots,9,11,13,15$;
\item if $p=3$, then $t\ne 1,2,\cdots,13,15,16,19$;
\item If $p>3$, then $t\ne 1,2,\cdots,13,15,16,17,19\ (\text{if $p\ne5$}),20,21\ (\text{if $p\ne7$})$.
\end{enumerate}
\end{corollary}
\begin{proof}
We only prove (2) here, since the proof for (1),(2) and (3) are similar.

In (2), with $p=3$, the lower bound of $\dim H$ by Theorem \ref{thm1} is $14p$, which corresponds to two different minimal forms $L(p,2)$ and $L(p,3)$. Then $t\ne 1,2,\cdots,13$ are proved.

See that $L(p,4)=42p$. Hence, if $14p<\dim H< 42p$, then as $\Bbbk$-spaces, the block system of $H$ must be isomorphic to $L(p,2)$ or $L(p,3)$ with basic blocks added.

If $t=15,16,$ or $19$, it is not difficult to see that there are at least one basic block $L_{p,1,1}$ of $\dim p$ being added.

By Proposition \ref{prop3} (1), the block $B^{1,1}_m$ with $m= \max \{m'\ |\ B^{1,1}_{m'}\ne 0\}$ has a fixed dimension $p$. This means the new added basic block $L_{p,1,1}$ leads to a nonzero block $B^{1,1}_n$ of $H$ with $n<m$. Then for $l= \min \{m'\ |\ B^{1,1}_{m'}\ne 0\}$, $l<m$. By Proposition \ref{prop3} (2), there exists $B^{1,d}_{l'}\ne 0$, $B^{d,1}_{l'}\ne0$ and $B^{d,d}_{l'-1}\ne0$ for some $l'>l>1$ and $d>1$. Each of the three blocks is not one of the six necessary blocks of $H$. Then there are more basic blocks need to be added into $L(p,2)$ or $L(p,3)$.

For $L(p,2)$, $\dim H \ge \dim L(p,2) + \dim L(p,1,1)+ \dim L(p,2,1)+ \dim L(p,2,2)= 21p$;
for $L(p,3)$, $\dim H \ge \dim L(p,3) + \dim L(p,1,1)+ \dim L(p,3,1)+ \dim L(p,3,3)= 24p$.
Both contradict with $\dim H = tp$. So (2) is proved.
\end{proof}

Corollary \ref{cor5} provides a new angle to look at things we have already known. For example, the fact proved in \cite{DAIJIRO2011HOPF} that there is no non-cosemisimple Hopf algebra of dimension $ 30$ is a direct result from this corollary, Theorem \ref{thm1} and \cite[Corollary 2.2]{ng2005hopf}. If $H$ is a noncosemisimple Hopf algebra with $\dim H= 30$, by Theorem \ref{thm1}, $|G(H)|\ne 15,10,6$, and by Corollary \ref{cor5}, $|G(H)|\ne 5,3,2$. By \cite[Corollary 2.2]{ng2005hopf}, $H$ or $H*$ is not unimodular. Then $|G(H^*)|=30$ or $|G(H)|=30$, which contradicts with $H$ being non-coseimisimple.

\begin{remark}\label{rmk3}
Some of the results in Corollary \ref{cor5} have already been proved by different methods.

Under conditions of Corollary \ref{cor5}, let $q$ be a prime number with $p<q$. By \cite{ng2005hopf}, there is no non-cosemisimple Hopf algebra of dimension $2q$. By \cite{Andruskiewitsch1998Hopf}, a non-cosemisimple Hopf algebra of dimension $p^2$ is a Taft algebra, which is pointed. By \cite{ng2004hopf} and \cite{ng2008hopf}, Hopf algebras of dimension $pq$ with $p,q$ being odd primes and $p<q\le 4p+11$ are cosemisimple. Then all cases in (1)(2)(3) that $t$ does not equal a prime number are covered.

For $p=2$, $t\ne 4$ is proved in \cite{williams1988finite}, which shows that there is no non-cosemisimple non-pointed Hopf algebra of dimension dimension $8$; $t\ne 6$ is proved in \cite{natale2002hopf} that non-cosemisimple Hopf algebras of dimension $12$ always have nontrivial primitives; $t\ne 8$ is proved in \cite{garcia2010hopf}, in which Hopf algebras of dimension $16$ are classified completely; $t\ne 9$ is proved in \cite{hilgemann2009hopf}, which shows that there is no non-cosemisimple non-pointed Hopf algebra of dimension $2q^2$; and $t\ne15$ follows from \cite{DAIJIRO2011HOPF}, which completes the classification of Hopf algebras of dimension $30$.

For $p=3$, $t\ne 4,6$ follows from \cite{natale2002hopf} and \cite{hilgemann2009hopf}, and $t\ne 8,9,10$ is well known by \cite[Proposition 3.2]{beattie2013techniques}.
\end{remark}

Following Remark \ref{rmk1}, we list results in Corollary \ref{cor5} which have not been given by others as far as we know:
\begin{theorem}\label{thm2}
If $H$ is a non-cosemisimple Hopf algebra, then we have the following results.
\begin{enumerate}
\item For a prime number $p$,
    \begin{enumerate}
    \item if $\dim H= 12p$ with $p>3$, then $|G(H)|\ne p$;
    \item if $\dim H= 15p$ with $p>5$, then $|G(H)|\ne p$;
    \item if $\dim H= 16p$ with $p\ge 3$, then $|G(H)|\ne p$;
    \item if $\dim H= 20p$ with $p>5$, then $|G(H)|\ne p$;
    \item if $\dim H= 21p$ with $p=5$ or $p>7$, then $|G(H)|\ne p$.
    \end{enumerate}
\item Assuming that $H$ has no nontrivial skew-primitives,
    \begin{enumerate}
    \item if $\dim H= 36\ \text{or}\ 45$, then $|G(H)|\ne3$;
    \item if $\dim H=75\ \text{or}\ 100$, then  $|G(H)|\ne 5$.
    \end{enumerate}

\end{enumerate}
\end{theorem}


\vskip7mm

{\it Acknowledgments.}  This research is supported by the NSFC (Grants No. 11271319).

\vskip7mm

\bibliographystyle{plain}
\bibliography{mybibtex}
\end{document}